\documentclass{amsart}
\usepackage{amssymb}
\usepackage[all,cmtip]{xy}
\usepackage{graphicx}
\usepackage[usenames,dvipsnames]{color}

\usepackage{enumerate}
\newtheorem{theorem}{Theorem}%[section]

\newtheorem{lemma}[theorem]{Lemma}

\newtheorem{remark}[theorem]{Remark}

\newcommand\calB{{\mathcal B}}

\newcommand\calF{{\mathcal F}}

\newcommand\calI{{\mathcal I}}

\newcommand\calO{{\mathcal O}}

\newcommand\calT{{\mathcal T}}

\newcommand {\Z}{\mathbb{Z}}

\usepackage{amssymb}
\usepackage[all,cmtip]{xy}
\usepackage{graphicx}
\usepackage[usenames,dvipsnames]{color}

\newcommand{\flip}{\operatorname{flip}}

\usepackage{url}
\usepackage{hyperref}

\usepackage{tikz}
\usetikzlibrary{decorations.markings}
\usepackage{xcolor}
\definecolor{light-gray}{gray}{0.8}

\usepackage{subfig}
\usetikzlibrary{arrows}
\usetikzlibrary{calc}

\begin{document}
\title{A new simple proof of the Aztec diamond theorem}

\author{Manuel Fendler}
\address{Institut f\"ur Mathematik, Carl von Ossietzky Universit\"at Oldenburg, 26111 Oldenburg, Germany}
\email{manuel.fendler@uni-oldenburg.de}

\author{Daniel Grieser}
\address{Institut f\"ur Mathematik, Carl von Ossietzky Universit\"at Oldenburg, 26111 Oldenburg, Germany}
\email{daniel.grieser@uni-oldenburg.de}

\date{\today}
\subjclass[2010]{
05A15 % Combinatorics, exact enumeration problems
% 35P99,   %	None of the above, but in this section (PDE: Spectral Theory and eigenvalue problems)
%35R30,   %	Inverse problems
%51N20   %	Euclidean analytic geometry
}

\begin{abstract}
The Aztec diamond of order $n$ is the union of lattice squares in the plane intersecting the square $|x|+|y|<n$. 
The Aztec diamond theorem states that the number of domino tilings of this shape is $2^{n(n+1)/2}$. It was first proved by Elkies, Kuperberg, Larsen and Propp in 1992. We give a new simple proof of this theorem.
\end{abstract}

\maketitle

A {\em domino} is a $1\times2$ rectangle in the plane whose corners are lattice points, i.e.\ have integer coordinates. A {\em domino tiling} of a subset $S$ of the plane is a covering of $S$ by a set of dominoes whose interiors are disjoint. 
The problem of counting the number of domino tilings (or tilings, for short) of a given set $S$ has received much attention in the last 50 years, partly because of its significance in physics, but also because of the many beautiful mathematical structures that have appeared in its study.
For example, the number of domino tilings of a $2\times n$ rectangle is the $n$th Fibonacci number, and there is a rather non-trivial formula for the number of domino tilings of an $m\times n$-rectangle due to Kasteleyn \cite{Kas:SDLINDAQL}, Fisher and Temperley \cite{TemFis:DPSMER}, which for even $m=n$ reads 
$ 2^{n^2/2} \prod\limits_{j,k=1}^{n/2} \left(\cos^2 \frac{j\pi}{n+1} + \cos^2 \frac{k\pi}{n+1}\right)\,. $
One of the amazing facts in this area is that when essentially rotating the square by $45$ degrees the number of tilings is given by a much simpler formula.

More precisely, define the {\em Aztec diamond of order} $n$ by 
$$ A_n = \bigcup \{Q:\, Q\cap \{(x,y):\, |x|+|y|<n\}\neq\emptyset\}$$
where $Q$ ranges over squares $[k,k+1]\times[l,l+1]$ with $k,l\in\Z$. See Figure \ref{aztec diamond}.
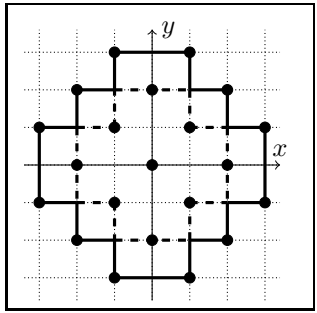
\begin{figure}[ht]
\centering
\fbox{
\begin{tikzpicture}[scale=1]
\draw[->] (-.2,0) -- (3.2,0) node[above] {$x$};
\draw[->] (1.5,-1.8) -- (1.5,1.8) node[right] {$y$};  
\draw[very thick](0,0.5)--(0,-0.5);
\draw[very thick](0,-0.5)--(0.5,-0.5);
\draw[very thick](0.5,-0.5)--(0.5,-1);
\draw[very thick](0.5,-1)--(1,-1);
\draw[very thick](1,-1)--(1,-1.5);
\draw[very thick](1,-1.5)--(2,-1.5);
\draw[very thick](2,-1.5)--(2,-1);
\draw[very thick](2,-1)--(2.5,-1);
\draw[very thick](2.5,-1)--(2.5,-0.5);
\draw[very thick](2.5,-0.5)--(3,-0.5);
\draw[very thick](3,-0.5)--(3,0.5);
\draw[very thick](3,0.5)--(2.5,0.5);
\draw[very thick](2.5,0.5)--(2.5,1);
\draw[very thick](2.5,1)--(2,1);
\draw[very thick](2,1)--(2,1.5);
\draw[very thick](2,1.5)--(1,1.5);
\draw[very thick](1,1.5)--(1,1);
\draw[very thick](1,1)--(0.5,1);
\draw[very thick](0.5,1)--(0.5,0.5);
\draw[very thick](0.5,0.5)--(0,0.5);
\draw[very thick,dashed](.5,.5)--(.5,-.5);
\draw[very thick,dashed](.5,.5)--(1,.5);
\draw[very thick,dashed](.5,-.5)--(1,-.5);
\draw[very thick,dashed](1,.5)--(1,1);
\draw[very thick,dashed](1,-.5)--(1,-1);
\draw[very thick,dashed](1,1)--(2,1);
\draw[very thick,dashed](1,-1)--(2,-1);
\draw[very thick,dashed](2,-1)--(2,-.5);
\draw[very thick,dashed](2,1)--(2,.5);
\draw[very thick,dashed](2,.5)--(2.5,.5);
\draw[very thick,dashed](2,-.5)--(2.5,-.5);
\draw[very thick,dashed](2.5,.5)--(2.5,-.5);
\foreach \i in {0,.5,1,1.5,2,2.5,3}
\draw[densely dotted] (\i,-1.8) -- (\i,1.8);
\foreach \i in {-1.5,-1,-.5,0,.5,1,1.5}
\draw[densely dotted] (-.2,\i) -- (3.2,\i);

\foreach \i/\j in {0/-.5,0/.5, .5/-1,.5/0,.5/1, 1/-1.5,1/-.5,1/.5,1/1.5, 1.5/-1,1.5/0,1.5/1, 2/-1.5,2/-.5,2/.5,2/1.5, 2.5/-1,2.5/0,2.5/1, 3/-.5,3/.5}
\draw[fill=black] (\i,\j) circle (2pt);
\end{tikzpicture}
}
\caption{The Aztec Diamonds of orders 2 (dashed lines) and 3 (heavy lines). The marked points are nodes ($n=2$).}
\label{aztec diamond}
\end{figure}

\begin{theorem}[\cite{ElkKupLarPro:ASMDTI}] \label{theorem}
Denote by $T_n$ the number of domino tilings of $A_n$. Then  
$$ T_n= 2^{n(n+1)/2}\,.$$ 
\end{theorem}
See also Remark \ref{remark} for a refinement.
The first four proofs of this formula were given in \cite{ElkKupLarPro:ASMDTI}, \cite{ElkKupLarPro:ASMDTII}, later other proofs appeared in  \cite{BruKir:ADDHDSN}, \cite{EuFu:SPADT} (these two proofs are essentially identical) and \cite{Kok:DTADS}. % note: Brualdi_Kirk and Eu-Fu seem to have the same proof using Hankel det's of Schröder numbers.
See \cite{KenOko:WID} and references given there for a more refined discussion of domino tilings, and \cite{Lai:GADTI} for a recent generalization of the theorem.
\medskip

In this note we give a proof of this theorem which is inspired by the first proof in \cite{ElkKupLarPro:ASMDTI}, but simplifies it in various respects, for example we replace the height function arguments used there by a very simple direct construction, see Lemma \ref{lem1}. 

We use induction on $n$. Clearly $T_1=2$, so in order to prove the theorem it suffices to show that
\begin{equation}
 \label{eqn:ind step}
T_{n+1} = 2^{n+1} T_n\,.
\end{equation}
We need a few preparations before we give the proof of this recursion.
We fix $n$ throughout. The strategy is to associate a 'field of arrows' to each tiling, and then to relate the fields of arrows arising from $A_n$ tilings to those arising from $A_{n+1}$ tilings.
\subsubsection*{Nodes and lattice squares}
Consider the two Aztec diamonds $A_n\subset A_{n+1}$. We call a lattice point $(i,j)\in A_{n+1}$ satisfying $i+j\equiv n\mod 2$ a {\em node}. In particular, the extreme points of $A_{n+1}$ are nodes. All other nodes are contained in $A_n$ and are called {\em interior nodes}. See Figure \ref{aztec diamond} for the case $n=2$. A {\em lattice square} will be a $1\times1$ square contained in $A_{n+1}$ whose corners are lattice points. A {\em boundary square} is a lattice square contained in $A_{n+1}$ but not in $A_n$.
Notice that each lattice square has exactly two corners which are nodes, and each interior node is adjacent to four lattice squares.

\subsubsection*{Fields of arrows}
%We will associate fields of arrows to tilings of $A_n$ and of $A_{n+1}$. 
Suppose in each lattice square of $A_{n+1}$ we draw an arrow pointing from one corner node to the other.
We call this collection of arrows a {\em field of arrows} if it satisfies the following condition. 
\begin{quote}
 \emph{\textbf{Arrow field condition:} 
Each interior node $N$ is either:
\begin{itemize}
 \item {\em attracting:} all adjacent arrows point towards $N$, or
 \item {\em repelling:} all adjacent arrows point away from $N$, or
 \item {\em transient:} any two collinear arrows adjacent to $N$ point in the same direction. 
\end{itemize}
 }
\end{quote}
See Figure \ref{local arrow patterns} for the six possible local arrow patterns. 
We call a field of arrows {\em pointing outward/inward} if all arrows in boundary squares point outward/inward, see Figure \ref{tilings+arrows} (B) and (C) for examples.

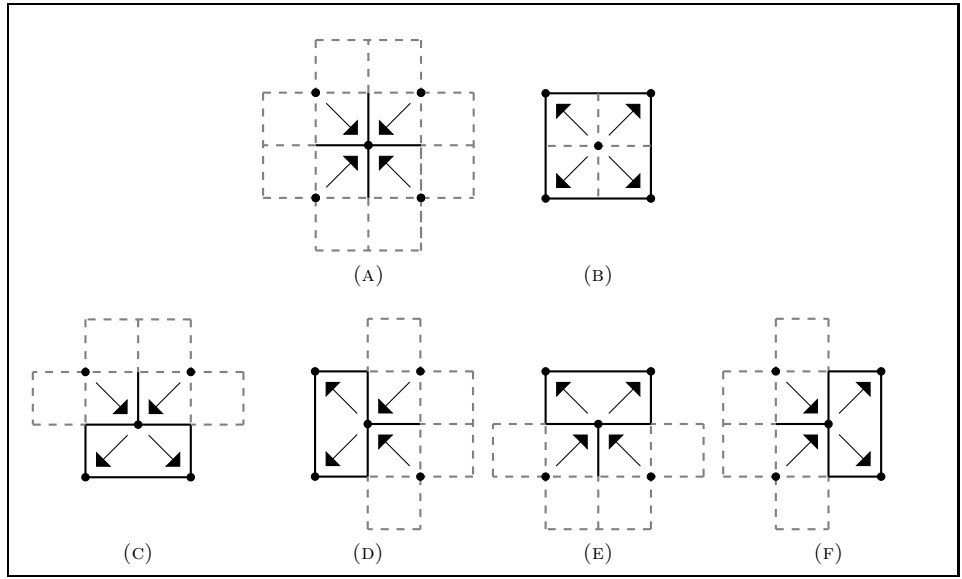
\begin{figure}[ht]
\fbox{\parbox{\dimexpr \linewidth - 2\fboxrule - 2\fboxsep}{\centering
\subfloat[]{
\begin{tikzpicture}[scale=.7]
\draw[-{triangle 90}](-.8,-.8)--(-.2,-.2);
\draw[-{triangle 90}](.8,.8)--(.2,.2);
\draw[-{triangle 90}](-.8,.8)--(-.2,.2);
\draw[-{triangle 90}](.8,-.8)--(.2,-.2);

\draw[thick] (1,0)--(-1,0);
\draw[thick] (0,1)--(0,-1);

\draw[fill=black] (0,0) circle (2pt);
\draw[fill=black] (1,1) circle (2pt);
\draw[fill=black] (-1,-1) circle (2pt);
\draw[fill=black] (-1,1) circle (2pt);
\draw[fill=black] (1,-1) circle (2pt);

\draw[gray,thick,dashed](0,1)--(0,2);
\draw[gray,thick,dashed](-1,2)--(1,2);
\draw[gray,thick,dashed](1,2)--(1,-2);
\draw[gray,thick,dashed](-2,1)--(2,1);
\draw[gray,thick,dashed](2,1)--(2,-1);
\draw[gray,thick,dashed](2,0)--(1,0);
\draw[gray,thick,dashed](0,-1)--(0,-2);
\draw[gray,thick,dashed](-1,-2)--(1,-2);
\draw[gray,thick,dashed](1,-2)--(1,0);
\draw[gray,thick,dashed](-2,-1)--(2,-1);
\draw[gray,thick,dashed](-1,2)--(-1,-2);
\draw[gray,thick,dashed](-2,1)--(-2,-1);
\draw[gray,thick,dashed](-1,0)--(-2,0);

\draw[fill=black] (0,0) circle (2pt);
\draw[fill=black] (1,1) circle (2pt);
\draw[fill=black] (-1,-1) circle (2pt);
\draw[fill=black] (-1,1) circle (2pt);
\draw[fill=black] (1,-1) circle (2pt);
\end{tikzpicture}
}
\subfloat[]{
\begin{tikzpicture}[scale=.7]
\draw[white](-2,0)--(2,0);
\draw[white](0,-2)--(0,2);

\draw[-{triangle 90}](-.2,-.2)--(-.8,-.8);
\draw[-{triangle 90}](.2,.2)--(.8,.8);
\draw[-{triangle 90}](-.2,.2)--(-.8,.8);
\draw[-{triangle 90}](.2,-.2)--(.8,-.8);

\draw[thick] (1,1)--(-1,1);
\draw[thick] (1,1)--(1,-1);
\draw[thick] (1,-1)--(-1,-1);
\draw[thick] (-1,1)--(-1,-1);
\draw[thick,gray,dashed] (1,0)--(-1,0);
\draw[thick,gray,dashed] (0,1)--(0,-1);

\draw[fill=black] (0,0) circle (2pt);
\draw[fill=black] (1,1) circle (2pt);
\draw[fill=black] (-1,-1) circle (2pt);
\draw[fill=black] (-1,1) circle (2pt);
\draw[fill=black] (1,-1) circle (2pt);
\end{tikzpicture}
}

\subfloat[]{
\begin{tikzpicture}[scale=.7]
\draw[white](-2,0)--(2,0);
\draw[white](0,-2)--(0,2);

\draw[{triangle 90}-](-.8,-.8)--(-.2,-.2);
\draw[-{triangle 90}](.8,.8)--(.2,.2);
\draw[-{triangle 90}](-.8,.8)--(-.2,.2);
\draw[{triangle 90}-](.8,-.8)--(.2,-.2);

\draw[thick] (1,0)--(-1,0);
\draw[thick] (0,1)--(0,0);
\draw[thick] (1,-1)--(-1,-1);
\draw[thick] (1,-1)--(1,0);
\draw[thick] (-1,-1)--(-1,0);

\draw[gray,dashed,thick](0,2)--(0,1);
\draw[gray,dashed,thick](1,2)--(-1,2);
\draw[gray,dashed,thick](2,1)--(-2,1);
\draw[gray,dashed,thick](1,2)--(1,0);
\draw[gray,dashed,thick](-1,2)--(-1,0);
\draw[gray,dashed,thick](2,1)--(2,0);
\draw[gray,dashed,thick](-2,1)--(-2,0);
\draw[gray,dashed,thick](2,0)--(1,0);
\draw[gray,dashed,thick](-2,0)--(-1,0);

\draw[fill=black] (0,0) circle (2pt);
\draw[fill=black] (1,1) circle (2pt);
\draw[fill=black] (-1,-1) circle (2pt);
\draw[fill=black] (-1,1) circle (2pt);
\draw[fill=black] (1,-1) circle (2pt);
\end{tikzpicture}
}
\subfloat[]{
\begin{tikzpicture}[scale=.7]
\draw[white](-2,0)--(2,0);
\draw[white](0,-2)--(0,2);

\draw[{triangle 90}-](-.8,-.8)--(-.2,-.2);
\draw[-{triangle 90}](.8,.8)--(.2,.2);
\draw[{triangle 90}-](-.8,.8)--(-.2,.2);
\draw[-{triangle 90}](.8,-.8)--(.2,-.2);

\draw[thick] (1,0)--(0,0);
\draw[thick] (0,1)--(0,-1);
\draw[thick] (-1,1)--(-1,-1);
\draw[thick] (-1,-1)--(0,-1);
\draw[thick] (0,1)--(-1,1);

\draw[gray,dashed,thick](2,0)--(1,0);
\draw[gray,dashed,thick](2,1)--(2,-1);
\draw[gray,dashed,thick](1,2)--(1,-2);
\draw[gray,dashed,thick](2,1)--(0,1);
\draw[gray,dashed,thick](2,-1)--(0,-1);
\draw[gray,dashed,thick](1,2)--(0,2);
\draw[gray,dashed,thick](1,-2)--(0,-2);
\draw[gray,dashed,thick](0,2)--(0,1);
\draw[gray,dashed,thick](0,-2)--(0,-1);

\draw[fill=black] (0,0) circle (2pt);
\draw[fill=black] (1,1) circle (2pt);
\draw[fill=black] (-1,-1) circle (2pt);
\draw[fill=black] (-1,1) circle (2pt);
\draw[fill=black] (1,-1) circle (2pt);
\end{tikzpicture}
}
\subfloat[]{
\begin{tikzpicture}[scale=.7]
\draw[white](-2,0)--(2,0);
\draw[white](0,-2)--(0,2);

\draw[-{triangle 90}](-.8,-.8)--(-.2,-.2);
\draw[{triangle 90}-](.8,.8)--(.2,.2);
\draw[{triangle 90}-](-.8,.8)--(-.2,.2);
\draw[-{triangle 90}](.8,-.8)--(.2,-.2);

\draw[thick] (1,0)--(-1,0);
\draw[thick] (0,0)--(0,-1);
\draw[thick] (1,1)--(-1,1);
\draw[thick] (-1,1)--(-1,0);
\draw[thick] (1,1)--(1,0);

\draw[gray,dashed,thick](0,-2)--(0,-1);
\draw[gray,dashed,thick](1,-2)--(-1,-2);
\draw[gray,dashed,thick](2,-1)--(-2,-1);
\draw[gray,dashed,thick](1,-2)--(1,0);
\draw[gray,dashed,thick](-1,-2)--(-1,0);
\draw[gray,dashed,thick](2,-1)--(2,0);
\draw[gray,dashed,thick](-2,-1)--(-2,0);
\draw[gray,dashed,thick](2,0)--(1,0);
\draw[gray,dashed,thick](-2,0)--(-1,0);

\draw[fill=black] (0,0) circle (2pt);
\draw[fill=black] (1,1) circle (2pt);
\draw[fill=black] (-1,-1) circle (2pt);
\draw[fill=black] (-1,1) circle (2pt);
\draw[fill=black] (1,-1) circle (2pt);
\end{tikzpicture}
}
\subfloat[]{
\begin{tikzpicture}[scale=.7]
\draw[white](-2,0)--(2,0);
\draw[white](0,-2)--(0,2);

\draw[-{triangle 90}](-.8,-.8)--(-.2,-.2);
\draw[{triangle 90}-](.8,.8)--(.2,.2);
\draw[-{triangle 90}](-.8,.8)--(-.2,.2);
\draw[{triangle 90}-](.8,-.8)--(.2,-.2);

\draw[gray,dashed,thick](-2,0)--(-1,0);
\draw[gray,dashed,thick](-2,1)--(-2,-1);
\draw[gray,dashed,thick](-1,2)--(-1,-2);
\draw[gray,dashed,thick](-2,1)--(0,1);
\draw[gray,dashed,thick](-2,-1)--(0,-1);
\draw[gray,dashed,thick](-1,2)--(0,2);
\draw[gray,dashed,thick](-1,-2)--(0,-2);
\draw[gray,dashed,thick](0,2)--(0,1);
\draw[gray,dashed,thick](0,-2)--(0,-1);

\draw[thick] (0,0)--(-1,0);
\draw[thick] (0,1)--(0,-1);
\draw[thick] (1,-1)--(1,1);
\draw[thick] (1,-1)--(0,-1);
\draw[thick] (1,1)--(0,1);

\draw[fill=black] (0,0) circle (2pt);
\draw[fill=black] (1,1) circle (2pt);
\draw[fill=black] (-1,-1) circle (2pt);
\draw[fill=black] (-1,1) circle (2pt);
\draw[fill=black] (1,-1) circle (2pt);
\end{tikzpicture}
}
}}

\caption{The six possible local arrow patterns, and how they arise from domino tilings. The precise position of the dominoes on dashed lines is inessential.}
\label{local arrow patterns}
\end{figure}

\subsubsection*{From tilings to fields of arrows}
We now associate fields of arrows to tilings. More precisely, we define an outward (resp.\ inward) pointing field of arrows $F(\calT)$ on $A_{n+1}$ for each tiling $\calT$ of $A_{n+1}$ (resp.\ $A_{n}$). 
%See Figure  for illustration.
This is done domino by domino by the rule indicated in Figure \ref{tilings+arrows}(A).\footnote{Note that the field of arrows is always defined on all of $A_{n+1}$, even for a tiling of $A_n$.}

More explicitly, consider first a tiling $\calT$ of $A_{n+1}$. Define $F(\calT)$ as follows:
Each domino in the tiling has exactly two corners which are nodes. The two arrows contained in that domino are chosen to point towards these corner nodes. Since each boundary domino has a corner node lying on the boundary of $A_{n+1}$, all arrows in boundary squares point outward. Also, 
by looking at the ways that dominoes can lie adjacent to any interior node, we see that the arrow field condition is satisfied. See Figure \ref{local arrow patterns}, and Figure \ref{tilings+arrows} (B) for an example.

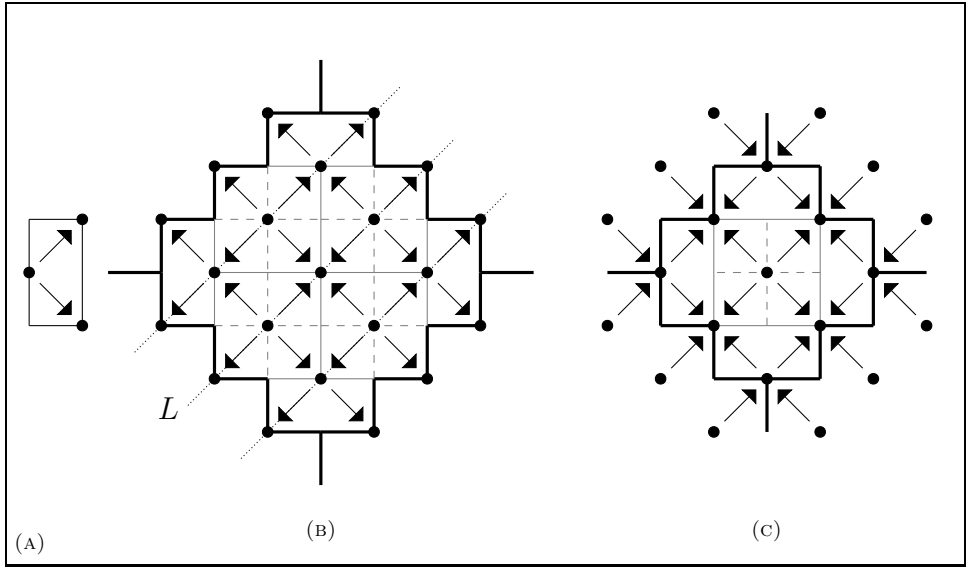
\begin{figure}[ht]
\fbox{\parbox{\dimexpr \linewidth - 2\fboxrule - 2\fboxsep}{\centering
\subfloat[]{
\begin{tikzpicture}[scale=1]
\foreach \i/\j in {0/0,.5/-.5}
 \draw[white,rotate=45]($(0,0)!\i!(-.1,-.1)$){(\i-2.2,\j-2.2)--(\i+2.2,\j+2.2)};

\draw[rotate=45](.5,.5)--(-.5,-.5);
\draw[rotate=45](.5,.5)--(1,0);
\draw[rotate=45](-.5,-.5)--(0,-1);
\draw[rotate=45](1,0)--(0,-1);

\draw[fill=black,rotate=45] (0,0) circle (2pt);
\draw[fill=black,rotate=45] (1,0)  circle (2pt);
\draw[fill=black,rotate=45] (0,-1)  circle (2pt);
\draw[-{triangle 90},rotate=45](.2,0)--(.8,0);
\draw[-{triangle 90},rotate=45](0,-.2)--(0,-.8);
\end{tikzpicture}
}
\subfloat[]{
\begin{tikzpicture}[scale=1]
\foreach \i/\j in {0/0,.5/-.5,1/-1,1.5/-1.5,2/-2,2.5/-2.5,3/-3,3.5/-3.5,4/-4}
 \draw[white,rotate=45]($(0,0)!\i!(-.1,-.1)$){(\i-2.2,\j-2.2)--(\i+2.2,\j+2.2)};

\foreach \i in {1,2,3,4}
 \draw[very thick,rotate=45] ($(0,0)!\i!(-.1,-.1)$){(\i,0)--(\i-.5,-.5)};
\foreach \i in {0,1,2,3}
 \draw[very thick,rotate=45] ($(0,0)!\i!(-.1,-.1)$){(\i,0)--(\i+.5,-.5)};
\foreach \i in {-1,-2,-3}
 \draw[very thick,rotate=45] ($(0,0)!\i!(-.1,-.1)$){(0,\i)--(.5,\i-.5)};
\foreach \i in {-1,-2,-3,-4}
 \draw[very thick,rotate=45] ($(0,0)!\i!(-.1,-.1)$){(0,\i)--(.5,\i+.5)};
\foreach \i in {1,2,3}
 \draw[very thick,rotate=45] ($(0,0)!\i!(-.1,-.1)$){(\i+.5,-3.5)--(\i,-4)};
\foreach \i in {0,1,2,3}
 \draw[very thick,rotate=45] ($(0,0)!\i!(-.1,-.1)$){(\i+.5,-3.5)--(\i+1,-4)};
\foreach \i in {-1,-2,-3}
 \draw[very thick,rotate=45] ($(0,0)!\i!(-.1,-.1)$){(3.5,\i+.5)--(4,\i)};
\foreach \i in {-2,-3,-4}
 \draw[very thick,rotate=45] ($(0,0)!\i!(-.1,-.1)$){(3.5,\i+.5)--(4,\i+1)};
\draw[{triangle 90}-,rotate=45](2,-1.8)--(2,-1.2);
\draw[{triangle 90}-,rotate=45](1.2,-1)--(1.8,-1);
\draw[-{triangle 90},rotate=45](3,-1.8)--(3,-1.2);
\draw[-{triangle 90},rotate=45](2.2,-1)--(2.8,-1);
\draw[{triangle 90}-,rotate=45](1,-1.2)--(1,-1.8);
\draw[{triangle 90}-,rotate=45](1.8,-2)--(1.2,-2);
\draw[-{triangle 90},rotate=45](1,-2.2)--(1,-2.8);
\draw[-{triangle 90},rotate=45](1.8,-3)--(1.2,-3);
\draw[-{triangle 90},rotate=45](2,-2.8)--(2,-2.2);
\draw[-{triangle 90},rotate=45](2.8,-2)--(2.2,-2);
\draw[-{triangle 90},rotate=45](3,-2.2)--(3,-2.8);
\draw[-{triangle 90},rotate=45](2.2,-3)--(2.8,-3);
\draw[-{triangle 90},rotate=45](1,-.8)--(1,-.2);
\draw[-{triangle 90},rotate=45](2,-.8)--(2,-.2);
\draw[-{triangle 90},rotate=45](3,-.8)--(3,-.2);
\draw[-{triangle 90},rotate=45](.8,-1)--(.2,-1);
\draw[-{triangle 90},rotate=45](.8,-2)--(.2,-2);
\draw[-{triangle 90},rotate=45](.8,-3)--(.2,-3);
\draw[-{triangle 90},rotate=45](3.2,-1)--(3.8,-1);
\draw[-{triangle 90},rotate=45](3.2,-2)--(3.8,-2);
\draw[-{triangle 90},rotate=45](3.2,-3)--(3.8,-3);
\draw[-{triangle 90},rotate=45](1,-3.2)--(1,-3.8);
\draw[-{triangle 90},rotate=45](2,-3.2)--(2,-3.8);
\draw[-{triangle 90},rotate=45](3,-3.2)--(3,-3.8);
\draw[dashed,gray,rotate=45](1.5,-.5)--(2.5,-1.5);
\draw[gray,rotate=45](2.5,-.5)--(3.5,-1.5);
\draw[dashed,gray,rotate=45](2.5,-.5)--(1.5,-1.5);
\draw[gray,rotate=45](3,-1)--(1,-3);
\draw[dashed,gray,rotate=45](.5,-1.5)--(1.5,-2.5);
\draw[gray,rotate=45](.5,-2.5)--(1.5,-3.5);
\draw[gray,rotate=45](2,-2)--(3,-3);
\draw[gray,rotate=45](2.5,-3.5)--(3.5,-2.5);
\draw[gray,rotate=45](.5,-1.5)--(1.5,-.5); 
\draw[dashed,gray,rotate=45](1.5,-2.5)--(2.5,-3.5);
\draw[dashed,gray,rotate=45](1.5,-3.5)--(2.5,-2.5);
\draw[dashed,gray,rotate=45](2.5,-1.5)--(3.5,-2.5);
\draw[dashed,gray,rotate=45](2.5,-2.5)--(3.5,-1.5);
\draw[gray,rotate=45](1,-1)--(2,-2);
\draw[dashed,gray,rotate=45](1.5,-1.5)--(.5,-2.5);
\foreach \i/\j in {1/0,2/0,3/0,0/-1,1/-1,2/-1,3/-1,4/-1,0/-2,1/-2,2/-2,3/-2,4/-2,0/-3,1/-3,2/-3,3/-3,4/-3,1/-4,2/-4,3/-4}
 \draw[fill=black,rotate=45] ($(0,0)!\i!(-.1,-.1)$){(\i,\j)} circle (2pt);
 
\node at (.8,-1.8) {{\Large $L$}};

\draw[densely dotted,rotate=45](-.5,-2)--(4.5,-2);
\draw[densely dotted,rotate=45](-.5,-1)--(4.5,-1);
\draw[densely dotted,rotate=45](-.5,-3)--(4.5,-3);

\end{tikzpicture}
}
\subfloat[]{
\begin{tikzpicture}[scale=1]
\foreach \i/\j in {0/0,.5/-.5,1/-1,1.5/-1.5,2/-2,2.5/-2.5,3/-3,3.5/-3.5,4/-4}
 \draw[white,rotate=45]($(0,0)!\i!(-.1,-.1)$){(\i-2.2,\j-2.2)--(\i+2.2,\j+2.2)};

\foreach \i in {1,2,3}
 \draw[very thick,rotate=45] ($(0,0)!\i!(-.1,-.1)$){(\i+.5,-.5)--(\i,-1)};
\foreach \i in {0,1,2}
 \draw[very thick,rotate=45] ($(0,0)!\i!(-.1,-.1)$){(\i+.5,-.5)--(\i+1,-1)};
\foreach \i in {-1,-2}
 \draw[very thick,rotate=45] ($(0,0)!\i!(-.1,-.1)$){(.5,\i-.5)--(1,\i-1)};
\foreach \i in {-1,-2,-3}
 \draw[very thick,rotate=45] ($(0,0)!\i!(-.1,-.1)$){(.5,\i-.5)--(1,\i)};
\foreach \i in {1,2}
 \draw[very thick,rotate=45] ($(0,0)!\i!(-.1,-.1)$){(\i+1,-3)--(\i+.5,-3.5)};
\foreach \i in {0,1,2}
 \draw[very thick,rotate=45] ($(0,0)!\i!(-.1,-.1)$){(\i+1,-3)--(\i+1.5,-3.5)};
\foreach \i in {-1,-2}
 \draw[very thick,rotate=45] ($(0,0)!\i!(-.1,-.1)$){(3,\i)--(3.5,\i-.5)};
\foreach \i in {-2,-3}
 \draw[very thick,rotate=45] ($(0,0)!\i!(-.1,-.1)$){(3,\i)--(3.5,\i+.5)};

\draw[gray,rotate=45](1,-2)--(2,-3);
\draw[gray,rotate=45](2,-1)--(3,-2);
\draw[gray,rotate=45](2,-3)--(3,-2);
\draw[dashed,gray,rotate=45](1.5,-2.5)--(2.5,-1.5);
\draw[dashed,gray,rotate=45](2.5,-2.5)--(1.5,-1.5);
\draw[gray,rotate=45](1,-2)--(2,-1);

\draw[{triangle 90}-,rotate=45](1,-1.8)--(1,-1.2);
\draw[-{triangle 90},rotate=45](1,-.2)--(1,-.8);
\draw[-{triangle 90},rotate=45](.2,-1)--(.8,-1);
\draw[{triangle 90}-,rotate=45](1.8,-1)--(1.2,-1);
\draw[-{triangle 90},rotate=45](2,-.2)--(2,-.8);
\draw[-{triangle 90},rotate=45](2.8,-1)--(2.2,-1);
\draw[-{triangle 90},rotate=45](3,-.2)--(3,-.8);
\draw[-{triangle 90},rotate=45](3.8,-1)--(3.2,-1);
\draw[-{triangle 90},rotate=45](1,-2.8)--(1,-2.2);
\draw[-{triangle 90},rotate=45](.2,-2)--(.8,-2);
\draw[-{triangle 90},rotate=45](1,-3.8)--(1,-3.2);
\draw[-{triangle 90},rotate=45](.2,-3)--(.8,-3);
\draw[{triangle 90}-,rotate=45](2,-1.2)--(2,-1.8);
\draw[{triangle 90}-,rotate=45](1.2,-2)--(1.8,-2);
\draw[-{triangle 90},rotate=45](2,-2.2)--(2,-2.8);
\draw[-{triangle 90},rotate=45](2,-3.8)--(2,-3.2);
\draw[-{triangle 90},rotate=45](1.2,-3)--(1.8,-3);
\draw[-{triangle 90},rotate=45](2.8,-3)--(2.2,-3);
\draw[-{triangle 90},rotate=45](3,-1.2)--(3,-1.8);
\draw[-{triangle 90},rotate=45](3,-2.8)--(3,-2.2);
\draw[-{triangle 90},rotate=45](2.2,-2)--(2.8,-2);
\draw[-{triangle 90},rotate=45](3.8,-2)--(3.2,-2);
\draw[-{triangle 90},rotate=45](3,-3.8)--(3,-3.2);
\draw[-{triangle 90},rotate=45](3.8,-3)--(3.2,-3);
\foreach \i/\j in {1/0,2/0,3/0,0/-1,1/-1,2/-1,3/-1,4/-1,0/-2,1/-2,2/-2,3/-2,4/-2,0/-3,1/-3,2/-3,3/-3,4/-3,1/-4,2/-4,3/-4}
 \draw[fill=black,rotate=45] ($(0,0)!\i!(-.1,-.1)$){(\i,\j)} circle (2pt);

\end{tikzpicture}
}}
}
\caption{(A) A single domino with nodes and arrows; 
(B) resp.\ (C) Outward resp.\ inward pointing field of arrows $(n=2)$, and  tilings of $A_{n+1}$ resp.\ $A_n$ from which they arise. In the $2\times2$ squares containing a dashed cross, the position of dominoes can be either horizontal or vertical.
}
\label{tilings+arrows}
\end{figure}

For a tiling $\calT$ of $A_n$ define $F(\calT)$ by putting arrows in the lattice squares of $A_n$ by the same rule as for $A_{n+1}$, and in addition putting inward pointing arrows in all boundary squares of $A_{n+1}$.
Then the arrow field condition is satisfied for the same reason as before, since the arrows in boundary squares may be thought of as arising from horizontal dominoes added along the outside boundary of $A_n$.

The core of the argument is the proof of the following lemma. Denote by 
$r(\calF)$ the number of repelling nodes for a field of arrows $\calF$.
\begin{lemma}[Number of tilings for a fixed field of arrows]\label{lem1}
Given an inward (resp.\ outward) pointing field of arrows $\calF$, there are precisely $2^{r(\calF)}$ domino tilings $\calT$ of $A_{n}$ (resp.\ $A_{n+1}$)
satisfying $F(\calT)=\calF$.
\end{lemma}
In particular, every outward (resp.\ inward) pointing field of arrows arises from some tiling of $A_{n+1}$ (resp.\ $A_{n}$).
\begin{proof}
First consider an outward pointing field of arrows $\calF$. 
For each lattice square, draw the two sides of the square which lie in the direction of its arrow
in bold face, see Figure \ref{fig4}.
Note that if $\calF$ comes from a tiling then these bold face lines must be boundary lines of dominoes.

The union $\calB$ of all these bold lines divides $A_{n+1}$ into components
(the connected components of $A_{n+1}\setminus \calB$).

\underline{Claim:} Each component is either a $1\times2$ rectangle or a $2\times2$ square, where the latter are precisely those $2\times2$ squares having repelling nodes at their center.

\underline{Proof of Claim:} Consider a lattice square $S$. W.l.o.g. assume that the lower left corner $P$ and the upper right corner $Q$ of $S$ are nodes, and that the arrow of $S$ points towards $P$. Since $\calF$ is outward pointing, the node $Q$ must be an interior node. It can be either repelling or transient. If it is repelling then the component containing $S$ is the $2\times2$ square centered at $Q$ (Figure \ref{fig4}(A)). If it is transient then there are two possibilities for the arrows adjacent to $Q$, and in both cases the component containing $S$ is a $1\times2$ rectangle (Figure \ref{fig4}(B) and (C)).
This proves the claim.
\smallskip

The claim implies that $\calF$ determines the tilings $\calT$ having $F(\calT)=\calF$, except for the choice of having two vertical or two horizontal dominoes in each $2\times2$ square. Conversely, any tiling
that fits into the decomposition
 will have $\calF$ as field of arrows. Since the number of $2\times2$ squares is $r(\calF)$, the lemma is proved in the case of outward pointing $\calF$.
\smallskip

Now consider an inward pointing field of arrows $\calF$. Define $\calB$ as above. Since $\calF$ is inward pointing, the boundary of $A_n$ is a subset of $\calB$. Now the same argument as above, only looking at components contained in $A_n$, shows that there are precisely $2^{r(\calF)}$ tilings of $A_n$ having $\calF$ as field of arrows.
\end{proof}
\begin{figure}[ht]
\fbox{\parbox{\dimexpr \linewidth - 2\fboxrule - 2\fboxsep}{\centering
\subfloat[]{
\begin{tikzpicture}[scale=2]
\draw[gray,thick](1,1)--(0,1);
\draw[gray,thick](1,1)--(1,0);
\draw[gray,thick](0,0)--(1,0);
\draw[gray,thick](0,0)--(0,1);

\node[gray] at(.1,.35){\LARGE $S$};
\draw[fill=black](0,0)circle(1pt);
\node at (-.1,-.1) {\normalsize $P$};
\node at (.4,.6) {$Q$};

\draw[{triangle 90}-](.1,.1)--(.4,.4);
\draw[very thick](0,0)--(.4,0);
\draw[very thick](0,0)--(0,.4);

\draw[-{triangle 90}](.6,.6)--(.9,.9);
\draw[very thick](1,1)--(.6,1);
\draw[very thick](1,1)--(1,.6);

\draw[-{triangle 90}](.6,.4)--(.9,.1);
\draw[very thick](1,0)--(.6,0);
\draw[very thick](1,0)--(1,.4);

\draw[-{triangle 90}](.3,.7)--(.1,.9);
\draw[very thick](0,1)--(.4,1);
\draw[very thick](0,1)--(0,.6);

\draw[gray,thick,dashed](0,.5)--(1,.5);
\draw[gray,thick,dashed](.5,0)--(.5,1);
\draw[fill=black](.5,.5)circle(1pt);
\end{tikzpicture}
}
\subfloat[]{
\begin{tikzpicture}[scale=2]
\draw[gray,thick,dashed](1,1)--(0,1);
\draw[gray,thick](1,.5)--(1,0);
\draw[gray,thick](0,0)--(1,0);
\draw[gray,thick](0,0)--(0,.5);
\draw[gray,thick,dashed](0,1)--(0,.5);
\draw[gray,thick,dashed](1,1)--(1,.5);

\node[gray] at(.1,.35){\LARGE $S$};
\draw[fill=black](0,0)circle(1pt);
\node at (-.1,-.1) {\normalsize $P$};
\node at (.4,.6) {$Q$};

\draw[{triangle 90}-](.1,.1)--(.4,.4);
\draw[very thick](0,0)--(.4,0);
\draw[very thick](0,0)--(0,.4);

\draw[{triangle 90}-](.6,.6)--(.9,.9);
\draw[very thick](.5,.5)--(.5,.9);
\draw[very thick](.5,.5)--(.9,.5);

\draw[-{triangle 90}](.6,.4)--(.9,.1);
\draw[very thick](1,0)--(.6,0);
\draw[very thick](1,0)--(1,.4);

\draw[{triangle 90}-](.3,.7)--(.1,.9);
\draw[very thick](.1,.5)--(.5,.5);

\draw[gray,thick](0,.5)--(1,.5);
\draw[gray,thick](.5,.5)--(1,.5);
\draw[gray,thick](.5,.5)--(.5,1);
\draw[fill=black](.5,.5)circle(1pt);
\end{tikzpicture}
}
\subfloat[]{
\begin{tikzpicture}[scale=2]
\draw[gray,thick](.5,1)--(0,1);
\draw[gray,thick,dashed](1,1)--(1,0);
\draw[gray,thick](0,0)--(.5,0);
\draw[gray,thick](0,0)--(0,1);
\draw[gray,thick,dashed](1,0)--(.5,0);
\draw[gray,thick,dashed](.5,1)--(1,1);

\node[gray] at(.1,.35){\LARGE $S$};
\draw[fill=black](0,0)circle(1pt);
\node at (-.1,-.1) {\normalsize $P$};
\node at (.4,.6) {$Q$};

\draw[{triangle 90}-](.1,.1)--(.4,.4);
\draw[very thick](0,0)--(.4,0);
\draw[very thick](0,0)--(0,.4);

\draw[{triangle 90}-](.6,.6)--(.9,.9);
\draw[very thick](.5,.5)--(.5,.9);
\draw[very thick](.5,.5)--(.9,.5);

\draw[{triangle 90}-](.6,.4)--(.9,.1);
\draw[very thick](.5,.5)--(.5,.1);

\draw[-{triangle 90}](.3,.7)--(.1,.9);
\draw[very thick](0,1)--(.4,1);
\draw[very thick](0,1)--(0,.6);

\draw[gray,thick](.5,.5)--(1,.5);
\draw[gray,thick](.5,.5)--(1,.5);
\draw[gray,thick](.5,0)--(.5,1);
\draw[fill=black](.5,.5)circle(1pt);
\end{tikzpicture}
}
}
}
\caption{Determining components from fields of arrows, Lemma \ref{lem1}}
\label{fig4}
\end{figure}
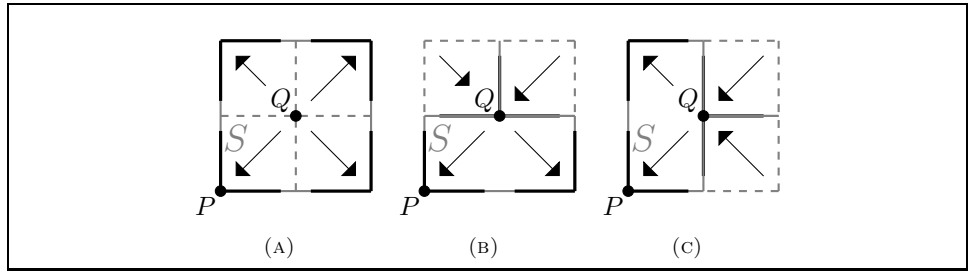

The other ingredient in the proof is the following lemma.
\begin{lemma}\label{lem2}
 For any outward pointing field of arrows on $A_{n+1}$ let $r$ be the number of repelling nodes and $a$ the number of attracting nodes. Then
 $$ r-a=n+1\,.$$
\end{lemma}
\begin{proof}
 The interior nodes of $A_{n+1}$ lie on $n+1$ lines running south-west to north-east, and on each line there are $n+2$ arrows. See Figure \ref{tilings+arrows}(B).

We claim that on each line there is one more repellent node than there are attracting nodes. Since there are $n+1$ lines, this will imply the lemma.

To prove the claim, fix one line $L$ and consider the sequence of arrows on $L$ and their changes of direction when traversing $L$ from south-west to north-east. Each arrow points forward (f) or backward (b). Changes b-f happens precisely at repelling nodes, and changes f-b at attracting nodes. There is no change of direction at transient nodes. Since the first arrow points b and the last arrow points f, there must be one more change b-f than f-b, which was to be shown.
\end{proof}

\subsubsection*{Proof of the recursion \eqref{eqn:ind step}}
Let $\calO$ (resp. $\calI$) be the sets of outward (inward) pointing fields of arrows on $A_{n+1}$. Reversing the direction of each arrow preserves the arrow field condition and therefore defines the map
\begin{equation}
 \label{eqn:flip}
 \flip: \calO\to \calI 
\end{equation}
which is its own inverse, hence bijective. 
For any $\calF\in\calO$, the attracting nodes of $\calF$ are the repelling nodes of 
$\flip(\calF)$, hence Lemma \ref{lem2} implies that $r(\calF) = r(\flip(\calF))+ n+1$. 
Lemma \ref{lem1} then implies that the number of tilings of 
$A_{n+1}$ corresponding to $\calF$ equals $2^{n+1}$ times the number of tilings of $A_n$ corresponding to $\flip(\calF)$. Summing over all $\calF\in\calO$ yields \eqref{eqn:ind step}.

Figures \ref{tilings+arrows} (B) and (C) show an example of an outward pointing field of arrows and its flip.

\begin{remark}\label{remark}
The flip defined in \eqref{eqn:flip} is the same map 
as the domino shuffling map defined in the fourth proof in \cite{ElkKupLarPro:ASMDTII}, although the definition of that map is different and requires proof of well-definedness, which is obvious in our setting.\footnote{In this correspondence, the nodes in this paper correspond to odd vertices in \cite{ElkKupLarPro:ASMDTII}.} Also, as shown there, it is straight-forward to prove a refinement of Theorem \ref{theorem}: The number of domino tilings of $A_n$ having exactly $2k$ horizontal dominoes is $$\binom{\frac{n(n+1)}{2}}{k}$$
for $k=0,\dots,\frac{n(n+1)}2$.
The central observation in the inductive proof of this fact is the following:
The numbers of horizontal $1\times 2$ components corresponding to a field $\calF\in\calO$ (as in the proof of Lemma \ref{lem1}) and to the field $\flip(\calF)$ are the same (in Figure \ref{tilings+arrows} (B) and (C) there are two such components). This is obvious from Figure \ref{fig4} (B) and (C) since the flip rotates these patterns by 180 degrees.
\end{remark}

%\bibliographystyle{plain}
%\bibliography{dglib}
\def\cprime{$'$} \def\cprime{$'$}

\end{document}